\documentclass[12pt, a4paper]{amsart}
\usepackage{amssymb}
\usepackage{amsmath}
\usepackage{bbm}
\usepackage{color}
\usepackage{verbatim}

\definecolor{darkred}{rgb}{0.75,0,0.3}

\newcommand\ab{\mathbf{a}}

\newcommand\AND{\quad\text{and}\quad}
\newcommand\bb{\mathbf{b}}
\newcommand\Bb{\mathbf{B}}
\newcommand\C{\mathbb C}
\newcommand\Cc{\boldsymbol \Delta} 

\newcommand\dps{\displaystyle}

\newcommand\Ff{\mathbb{F}}
\newcommand\fb{\mathbf{f}}
\newcommand\Fb{\mathbf{F}}
\newcommand\gb{\mathbf{g}}
\newcommand\HH{\mathcal{H}}
\newcommand\hor{\mathfrak{h}}

\newcommand\Mbd{\mathcal{M}} 
\newcommand\Mcpt{ \Cc_{\textrm{Mart},\,t} }    
\newcommand\Mcpr{ \Cc_{\textrm{Mart},\,\rho} }    

\newcommand\N{\mathbb N}
\newcommand\ol{\overline}

\newcommand\Pfi{\mathit{\Phi}}
\newcommand\Prob{\mathsf{Pr}}

\newcommand\Proj{\mathsf{Proj}}
\newcommand\Q{\mathbb{Q}}
\newcommand\R{\mathbb R}

\newcommand\Rbd{\mathcal{R}}
\newcommand\Rcp{\Cc_{\textrm{ratio}}}

\newcommand\st{\mathbf{s}}

\newcommand\supp{\operatorname{\sf supp}}
\newcommand\T{\mathbb{T}}

\newcommand\uno{\mathbf{1}}
\newcommand\vb{\mathbf{v}}
\newcommand\wb{\mathbf{w}}
\newcommand\wh{\widehat}
\newcommand\wt{\widetilde}
\newcommand\Xx{\mathcal{X}}
\newcommand\Z{\mathbb Z}

\numberwithin{equation}{section}

\newtheoremstyle{mythm}
  {9pt}
  {9pt}
  {\itshape}
  {0pt}
  {\bfseries}
  {}
  { }
  {\thmnumber{(#2)}\thmname{ #1}\thmnote{ #3}}

\newtheoremstyle{mydef}
  {9pt}
  {9pt}
  {\normalfont}
  {0pt}
  {\bfseries}
  {}
  { }
  {\thmnumber{(#2)}\thmname{ #1}\thmnote{ #3}}

\theoremstyle{mythm}
\newtheorem{thm}[equation]{Theorem.}
\newtheorem{pro}[equation]{Proposition.}
\newtheorem{lem}[equation]{Lemma.}
\newtheorem{cor}[equation]{Corollary.}

\theoremstyle{mydef}
\newtheorem{dfn}[equation]{Definition.}
\newtheorem{exa}[equation]{Example.}

\newtheorem{que}[equation]{Questions.}
\newtheorem{rmk}[equation]{Remark.}

\begin{document}$\,$ \vspace{-1truecm}
\title{Ratio limits and Martin boundary}
\author{\bf Wolfgang Woess}
\address{\parbox{.8\linewidth}{Institut f\"ur Diskrete Mathematik,\\ 
Technische Universit\"at Graz,\\
Steyrergasse 30, A-8010 Graz, Austria\\}}
\email{woess@tugraz.at}
\date{April 2021} 
\thanks{Supported by Austrian Science Fund projects FWF P31889-N35 and W1230} 
\subjclass[2020] {60G50; 
                  60J05}
\begin{abstract}
Consider an irreducible Markov chain which satisfies a ratio limit theorem, and let $\rho$
be the spectral radius of the chain. 
We investigate the relation of the the $\rho\,$-Martin boundary with the boundary induced by 
the $\rho\,$-harmonic kernel which appears in the ratio limit. Special emphasis is on random 
walks on non-amenable groups, specifically, free groups and hyperbolic groups. 
\end{abstract}

\maketitle

\markboth{{\sf W. Woess}}
{{\sf Ratio limits and Martin boundary}}
\baselineskip 15pt

\section{Introduction}\label{sec:intro}

Let $\Xx$  be the denumerable state space of a time-homogeneous 
Markov chain $(X_n)_{n \ge 0}$ with transition matrix $P$. We assume that
it is \emph{irreducible and aperiodic:} for any pair of points $x,y \in \Xx$, there is 
$k_{x,y} \in \N$ such that  $p^{(n)}(x,y) >0$ for all $n \ge k_{x,y}\,$, 
where $p^{(n)}(x,y)$ denotes the 
$n$-step transition probability from $x$ to $y$.
We will often (but not always) also 
assume that $P$ has   \emph{finite range,}  that is, for each $x \in \Xx$, there are only finitely many $y$ with $p(x,y) > 0$. (This simplifies parts of the technicalities.)
We choose and fix a \emph{root} $e \in \Xx$. In case of a group, this will be the identity.

In a variety of cases, part of which will be displayed below, 
one knows that a \emph{ratio limit theorem} holds, that is, there
is a function (kernel) $h : \Xx^2 \to (0,\infty)$ such that
\begin{equation}\label{eq:rl}
\lim_{n \to \infty} \frac{p^{(n)}(x,y)}{p^{(n)}(e,e)} = h(x,y) \quad \text{for all }\; x,y \in \Xx\,.
\end{equation}

Now let $\rho=\rho(P) = \limsup p^{(n)}(x,y)^{1/n}$ be the so-called 
\emph{spectral radius} of the Markov chain. It is independent of $x$ and $y$ by irreducibility.
If \eqref{eq:rl} holds then in a large variety of cases (including finite range) one deduces that 
\begin{equation}\label{eq:rhoH}
Ph = hP = \rho \cdot h
\end{equation}
in the sense of matrix products. In particular, for any fixed $y$, the function 
$x \mapsto  h(x,y)$ is a positive \emph{$\rho\,$-harmonic} function: more generally, for
real $t \ge \rho$, a function
$f: \Xx \to \R$ is $t$-harmonic, if
\begin{equation}\label{eq:rhoh}
\sum_w p(x,w) f(w) = t\,f(x) \quad \text{for all }\;x \in \Xx.
\end{equation}
When $t=1$, one just speaks of a harmonic function. Let us write 
$\HH^+(P,t)$ for the convex cone of positive $t$-harmonic functions.

For each $x$, we have $h(x,y) \le C_x\, h(e,y)$ for all $y$,
where $1/C_x = \rho^k\, p^{(k)}(e,x)$ with the smallest $k$ such that
this is non-zero. Therefore, we can consider the normalised \emph{ratio limit kernel}
$$
H(x,y) = h(x,y)/h(e,y),
$$
and the function $y \mapsto H(x,y)$ is bounded by $C_x$ for each $y$. 

\begin{dfn}\label{def:rlc}
The \emph{ratio limit compactification\footnote{$\;$\cite{Dor} calls this the \emph{full} ratio limit compactification.}} $\Rcp(\Xx)$ of
$\Xx$ is the (up to homomorphism) unique compact Hausdorff space
which contains $\Xx$ as a discrete, dense subset and has the following
properties: 
\begin{itemize}
 \item each function $H(x, \cdot)$ extends continuously to $\Rcp(\Xx)\,$,
and denoting the extended kernel also by $H$,  
 \item if $\xi, \eta \in \Rbd(\Xx) = \Rcp(\Xx) \setminus \Xx$ are distinct,
 then there is $x \in \Xx$ such that $H(x,\xi) \ne H(x,\eta)$.
\end{itemize}
\end{dfn}

For the general construction of such a compactification, within the present context 
see e.g. \cite[Thm. 7.13]{WMarkov}.
At least when $P$ has finite range, each function 
$x \mapsto H(x,\xi)$, $\xi \in \Rcp(\Xx)\,$, is $\rho\,$-harmonic.

\smallskip

In relation with his current work on Cuntz algebras related to random walks,
Adam {\sc Dor-On}, in an exchange that lead to the ``companion'' paper \cite{Dor},
has asked how one can describe this compactification in terms of a given 
structure of the underlying state space, and how 
it relates with the \emph{$\rho\,$-Martin compactification.} His questions were motivated by operator algebraic 
issues, and the answers provided here help with clarifying some of them. The Martin boundary will be 
recalled in the next section. 

\smallskip

The main interest is in the case when $\Xx = \Gamma$ is a finitely generated,
infinite group, and the Markov chain is a random walk induced by
a finitely supported probability measure $\mu$ on $\Gamma$, that is, 
$$
p(x,y)= \mu(x^{-1}y)\,.
$$
Irreducibility \& aperiodicity then amount to the property that the support of $\mu$ generates 
$\Gamma$ as a semi-group and that it is not contained in a coset of a proper normal
subgroup of $\Gamma$. Given the first of those two properties, the second will certainly
hold if $\mu(e) > 0$, where $e$ is the group identity.  (Assuming this is not a crucial restriction.)

If in this case we have a ratio limit theorem, then clearly $h(gx,gy) = h(x,y)$
for all $x,y,g \in \Gamma$, so that the function $f(x) = h(x,e)$ is $\rho\,$-harmonic, and
$h(x,y) = f(y^{-1}x)$. In that case, the ratio limit boundary is a $\Gamma$-space, and the 
elements of $\Gamma$ act continuously on $\Rcp(\Gamma)$.

\smallskip

For random walks on groups, a famous ratio limit theorem is due to {\sc  Avez}:

\begin{thm} [\cite{Av}]\label{thm:avez}
If the group $\Gamma$ is amenable and $\mu$ is symmetric, i.e., $\mu(x^{-1}) = \mu(x)$, 
then 
$$
\lim_{n \to \infty} \frac{p^{(n)}(x,y)}{p^{(n)}(e,e)} = 1 \quad \text{for all }\; x,y \in \Gamma\,.
$$
\end{thm}
(Finite support is not needed here.) In this case, we see that $\Rcp(\Xx)$ is the 
one-point compactification. We mention here that there are  symmetric random walks on amenable groups where the Martin boundary (and even the minimal one; see \S \ref{sec:martin}) 
is infinite. For example, for random walks on lamplighter groups over $\Z^d$ ($d \ge 3$),
the Poisson boundary is non-trivial, see {\sc Kaimanovich and Vershik}~\cite{KV}, whence there
are non-constant minimal harmonic functions. This is even true for certain lamplighter
random walks over $\Z$, see \cite{Wlamp}.

For non-amenable groups, there is a variety of results  which go beyond ratio limit theorems,
namely, \emph{local limit theorems} which provide an asymptotic evaluation of the
$n$-step transition probabilities. Typically (but not necessarily), they are of the form 
\begin{equation}\label{eq:local}
p^{(n)}(x,y) \sim C \,\, \beta(x,y)\, \rho^n\, n^{-\alpha} \quad \text{as }\;n \to \infty\,,
\end{equation}
where $\sim$ means that the quotient of the left and right hand sides tends to $1$,
and $C, \beta(x,y), \alpha > 0$.  
For results up to 2000, see \cite[Ch. III]{Wbook} and the references therein. 
The first of those results are due to {\sc Gerl}~\cite{Ge} and {\sc Sawyer}~\cite{Sa},
concerning random walks on free groups, resp. regular trees. The latest results
are by {\sc Gou\"ezel}~\cite{Gou1}, \cite{Gou2} and {\sc Dussaule} \cite{Dus}, for hyperbolic, resp. relatively hyperbolic groups. In between, there is a rather large body of work.  

\smallskip 

Whenever one has \eqref{eq:local}, this yields the ratio limit theorem with 
$h(x,y)= \beta(x,y)/\beta(e,e)$, and the kernel for the ratio limit compactification
is $H(x,y) = \beta(x,y) / \beta(e,y)$, where $e$ is the group identity.

\smallskip

In the present paper, we exhibit classes of random walks on non-amenable groups
where the ratio limit compactifictation coincides with the $\rho\,$-Martin compactification,
and one has a geometric description of the latter. Recall that $\rho<1$ on those groups
by {\sc Kesten}~\cite{Ke}.

After laying out some necessary 
preliminaries in \S \ref{sec:martin}, in \S \ref{sec:trees} we consider free groups and regular trees, and instead of stating and proving right away the most general result for general finite range random walks, we proceed step by step. Subsequently, in \S \ref{sec:hyperbolic}, 
we consider hyperbolic groups.  In \S \ref{sec:prod}, we discuss
how the ratio limit compactification interacts with taking direct or Cartesian products.
In \S \ref{sec:reduced}, we discuss the \emph{reduced} ratio limit boundary and show that in
the cases of \S \ref{sec:trees}, there is no reduction, while for hyperbolic groups, 
reduction may come only from a (finite) torsion subgroup.
The final \S \ref{sec:final} contains a brief discussion and an outlook on work 
to be done in the future.

It is on purpose that three different methods are displayed for (1) isotropic
random walks on trees, (2) finite range random walks on free groups, and (3) 
symmetric random walks on hyperolic groups (with finite range). In spite of
the fact that the most modern approach is based on the cited work of {\sc Gou\"ezel}~\cite{Gou1},
the older methods regarding (1), going back to {\sc Sawyer}~\cite{Sa}, resp.
(2), going back to {\sc Derriennic}~\cite{De} plus {\sc Lalley}~\cite{La}, are still
very valid in the author's view. Also, for the time being, (3) does not cover (1) and (2)
completely.

\medskip\noindent
\textbf{Acknowledgment.} The author thanks Adam Dor-On for posing the question,
for interesting exchanges on the subject, as well as for his comments on drafts of
this paper. Also, the author is extremely grateful to S\'ebastien Gou\"ezel for 
helpful input that facilitated the elaboration of the proof of Theorem
\ref{thm:hyp}. 

\section{Preliminaries; the Martin compactification}\label{sec:martin}

\noindent\textbf{A. The $t$-Martin compactification}
\\[5pt]
We let $t \ge \rho$ (focusing on the case when $t = \rho$). 
To describe the $t$-Martin compactification, consider for real $z > 0$ the 
\emph{Green function}
$$
G(x,y|z) = \sum_{n=0}^{\infty} p^{(n)}(x,y) \,z^n\,,\quad x,y \in \Xx\,.
$$
The radius of convergence of this power series is $r = 1/\rho$. At that value,
the series may either converge or diverge for all $x,y$. The latter is
the \emph{$\rho\,$-recurrent} case, the former the \emph{$\rho\,$-transient}
one. The quotient 
$$
F(x,y|z) = \frac{G(x,y|z)}{G(y,y,|z)} = \sum_{n=0}^{\infty} f^{(n)}(x,y) \,z^n
$$
is also a generating function: $f^{(n)}(x,y)$ is the probability that the 
Markov chain starting at $x$ first reaches $y$ at time $n \ge 0$. 
It is known \cite[Lemma 3.66]{WMarkov}  that $F(x,y|1/\rho)$ is always finite. 
In the case of a random walk on a group, it is known that one always has the
$\rho\,$-transient case, except when that group is virtually $\Z$ or $\Z^2$,
see e.g. {\sc Woess} \cite[Thm. 7.8]{Wbook}. 
For real $t \ge \rho$, 
the \emph{$t$-Martin kernel} is
$$
K(x,y|t) = \frac{F(x,y|1/t)}{F(e,y|1/t)}.
$$
Here we shall mainly work with $t = \rho$.
For fixed $y$, the function $x \mapsto K(x,y|\rho)$ satisfies \eqref{eq:rhoh}
at every $x \in \Xx \setminus y$. Furthermore, for fixed $x$, we have $K(x,\cdot|t) \le C_x$
with the same bound as above.  Thus, we can construct the compactification
$\Mcpt(\Xx)$ of $\Xx$, analogous to the one of Definition \ref{def:rlc}, for the $t$-Martin kernel $K(x,y|t)$ in the place
of $H(x,y)$. This is the $t$-Martin compactification, and $\Mbd_t(\Xx)
= \Mcpt(\Xx) \setminus \Xx$ is the \emph{$t$-Martin boundary.} Each function 
$x \mapsto K(x,\xi|t)$, $\xi \in \Mbd_t(\Xx)$, is $t$-harmonic. The significance of 
$\Mbd_t(\Xx)$ 
lies in the fact that every positive $t$-harmonic function $f$ on $\Xx$ has an integral 
representation 
$$
f(x) = \int_{\Mbd_t(\Xx)} K(x,\cdot|t)\, d\nu^f\,,
$$
where $\nu$ is a Borel measure on $\Mbd_t(\Xx)$. That measure is unique, if we require that
$\nu^f\bigl(\Mbd_t(\Xx) \setminus \Mbd_{t,\min}(\Xx)\bigr)=0$, where $\Mbd_{t,\min}(\Xx)$ is the \emph{minimal boundary:}
a positive $t$-harmonic function $f$ is called \emph{minimal,} if $f(e)=1$ and it cannot be written as
a convex combination of two distinct positive $t$-harmonic functions with value $1$ at $e$.
The minimal boundary is $\{ \xi \in \Mbd_t(\Xx) : K(\cdot,\xi|t) \text{ is minimal}\,\}$; it consists of all minimal $t$-harmonic functions.  

\smallskip

The finite range assumption simplifies a few things (for example, it yields that all 
extended Martin kernels are $t$-harmonic in the first variable), but is not needed for
the construction and properties. Also, aperiodicity is not needed, irreducibility
is sufficient, that is, for every pair $x,y \in \Xx$ there is $k = k_{x,y}$ such
that $p^{(k)}(x,y) > 0$. (It even suffices to have this only for $e$ and every $y$.)

\smallskip

For all these facts, see e.g. \cite[\S 24]{Wbook} or the 
references provided there. The largest part of the literature is on the ordinary $1$-Martin 
compactification. There is a simple tool to pass from the $t$-Martin compactification
to the latter. Namely, take one positive $t$-harmonic function $f$ and use the 
\emph{Doob transform:} define new transition probabilities by 
\begin{equation}\label{eq:doob}
p^f(x,y) = \frac{p(x,y)f(y)}{t\, f(x)}
\end{equation}
Then the $t$-Martin compactification of the original Markov chain is the 
$1$-Martin compactification of the Doob transform, and
$K_{P^f}(x,\cdot|1) = K(x,\cdot| t)/f(x)$. 

\medskip\goodbreak

\noindent\textbf{B. Comparing compactifications}
\\[5pt]
Let $\Cc(\Xx)$ and $\Cc'(\Xx)$ be two compactifications of the countable set $\Xx$,
that is, compact Hausdorff spaces which contain $\Xx$ as open and dense subsets. 
Then we say that the former is \emph{bigger} than the latter, if the identity
map on $\Xx$ extends to a continuous surjection $\Cc(\Xx) \to \Cc'(\Xx)$.
Clearly, if also $\Cc'(\Xx)$ is bigger than $\Cc(\Xx)$ then that continuous mapping
is a homeomorphism. In this case, we say that the two compactifications coincide \emph{geometrically.}

In typical cases, one has a ``natural'' geometric compactification of $\Xx$
and then wants to show that it coincides geometrically with the Martin compactification,
or that it is smaller than the Martin compactification.

In particular, we want to compare the ratio limit compactification with
the $\rho\,$-Martin compactification. In its geometric, resp. topological
meaning, this will mean that the latter is bigger than the former,
or that they coincide in the above sense. 

However, this is not the full picture; 
it does note yet fully reflect its analytic properties. 
One should also ask whether in that case the homoeomorphism 
$\tau: \Mcpr(\Xx) \to \Rcp(\Xx)$, in case it exists, is such that
\begin{equation}\label{eq:hom}
K(x,\xi|\rho) = H\bigl(x,\tau(\xi)\bigr) \quad \text{for all }\; x \in \Xx, \xi \in \Mbd_{\rho}(\Xx)\,.
\end{equation}
When we have two geometrically equal compactifications constructed
via two different kernels, let us say that they coincide \emph{analytically,} 
when the extended kernels coincide on the boundary. (``Coincide'' is of course meant
in terms of a homeomorphism as above.)

Indeed, no matter whether the ratio limit and $\rho\,$-Martin compactifications 
coincide geometrically or not:
whenever $\nu$ is a Borel measure on the ratio limit boundary $\Rbd$,
the function 
$$
f(x) = \int_{\Rbd} H(x,\cdot)\, d\nu
$$
is positive $\rho\,$-harmonic, and the above question basically amounts to asking whether
\emph{every} positive $\rho\,$-harmonic function arises in this way.

In general, it is for example known that for any two finite range, irreducible 
random walks on a free group, the Martin compactifications coincide geometrically,
but in general by no means analytically.

\begin{pro} \cite[Prop. 1]{PiWo2}\label{pro:compare}
Let $P, Q$ be two irreducible transition operators on $\Xx$ such that every positive 
harmonic function for $P$ is also harmonic for $Q$. Suppose that
\begin{itemize}
 \item the $1$-Martin compactifications of $P$ and $Q$ coincide geometrically, and
 \item the $1$-Martin boundary of $P$ is Dirichlet-regular, that is, every 
 continuous function on the boundary has a continuous extension to the entire
 compactification which is $1$-harmonic for $P$ on $\Xx$. 
\end{itemize}
Then the two compactifications coincide analytically, so that the positive $1$-harmonic
functions for $P$ and $Q$ coincide.
\end{pro}
 
We shall apply this below.

\section{Trees and free groups}\label{sec:trees}

Let $\T$ be the regular tree where each vertex has $q+1$ neighbours, with 
$q \ge 2$. 

If there is an edge between the vertices $x,y \in \T$ then we write  
$x \sim y$.
For any pair of vertices $x,y$, there is a unique finite geodesic $\pi(x,y)=[x=x_0\,,x_1\,,\dots, x_k=y]$ connecting the two, that is, 
$x_{i-1} \sim x_i$, and all $x_i$ are distinct. The distance between $x$ and $y$ is then
$d(x,y)=k$.
An infinite \emph{ray} is a 
one-sided infinite geodesic path $\pi = [x_0\,,x_1\,,x_2\,,\dots]$, that is, $x_k \sim x_{k-1}$
and all $x_k$ are distinct. Two rays $\pi$ and $\pi'$ are called \emph{equivalent,} if
they differ only by finitely many initial vertices. An \emph{end} of $\T$ is an 
equivalence class of rays. The \emph{geometric boundary} $\partial \T$ is the set of all ends.
For any $x \in \T$ and $\xi \in \partial \T$, there is a unique ray $\pi(x,\xi)$ starting at
$x$ which represents $\xi$. We set $\Cc_{\textrm{ends}}(\T) = \T \cup \partial \T\,$. 
 
We choose and fix a root vertex $e$. For $x \in \T$, we let $|x| = d(e,x)$. We define the \emph{branch} of $\Cc_{\textrm{ends}}(\T)$ at $x$ as
$$
\Cc_{\textrm{ends}}(\T_x) = \{ w \in \Cc_{\textrm{ends}}(\T) : x \in \pi(e,w) \}.
$$
The topology on $\Cc_{\textrm{ends}}(\T)$ is discrete on $\T$, while a neighbourhood basis of
$\xi \in \partial \T$ is given by the collection of all $\Cc_{\textrm{ends}}(\T_x)$, where 
$x \in \pi(e,\xi)$. Each of those sets is open and compact. 
This turns $\Cc_{\textrm{ends}}(\T)$ into a totally disconnected, compact Hausdorff space in which 
$\T$ is open and dense. It can be metricised as follows. For \emph{distinct}
$v, w \in \Cc_{\textrm{ends}}(\T)$, their \emph{confluent} $v \wedge w$ is the vertex
on $\pi(e,v) \cap \pi(e,w)$ furthest from the root $e$. Then
$$
\vartheta(v,w) = \begin{cases} q^{-|v \wedge w|}\,,&\text{if}\; v \ne w,\\ 
                            0\,,&\text{if}\; v = w 
              \end{cases}
$$ 
is an ultrametric which induces the above topology. Below, we shall also need the 
\emph{Busemann function} or \emph{horocycle index} with respect to an end 
$\xi \in \partial \T$. For a vertex $x \in \T$, this is
$$
\hor(x,\xi) = d(x, x \wedge \xi) - |x \wedge \xi| = \lim_{\T \ni y \to \xi}
d(x,y) - |y|.
$$

\medskip


\noindent
\textbf{A. Isotropic random walks on $\T$}
\\[3pt]
A Markov chain with transition matrix $P$ on $\T$ is called an \emph{isotropic random walk}, if
$p(x,y)$ depends only on $d(x,y)$. For $d \in \N$, let $P_d$ be the stochastic 
transition matrix with entries
$$
p_d(x,y) = \begin{cases} \dfrac{1}{(q+1)q^{d-1}}\,,&\text{if}\; d(x,y)=d\,,\\[3pt]
                         0\,,&\text{otherwise.} 
           \end{cases}              
$$
Also, we set $P_0 = I$. For each $d$, there is a polynomial $\wh P_d(t)$ of degree $d$ 
such that $P_d = \wh P_d(P_1)$. 
If $P$ is isotropic then it can be written as
a (possibly infinite) convex combination  $P= \sum_{d=0}^{\infty} a_d \,P_d\,$.
In order to guarantee irreducibility \& aperiodicity, we assume that 
$a_d > 0$ for some odd and some even $d$. We now refer to the results 
explained in \cite[\S 19.C]{Wbook}, due to \cite{Sa}.
The \emph{spherical transform} of $P$ is
$$
\wh P(t) = \sum_{d=0}^{\infty} a_d \,\wh P_d(t)\,.
$$
Thus, $P = \wh P(P_1)$.
It is very well known that 
\begin{equation}\label{eq:rhoP}
\rho(P_1)= \frac{2\sqrt{q}}{q+1} \AND \rho= \rho(P) = \wh P\bigl(\rho(P_1)\bigl). 
\end{equation}
Let 
\begin{equation}\label{eq:spherical}   
 \varphi(n) = \Bigl(1+\frac{q-1}{q+1}n\Bigr) q^{-n/2}\,, \quad n \in \N_0\,.
\end{equation}
Set $\Pfi(x,y) = \varphi\bigl(d(x,y)\bigr)$. Thus, $\Pfi(x)=\Pfi(x,e)$ is
the \emph{spherical function} which satisfies $P_1 \Pfi(x) = \rho(P_1)\,\Pfi(x)$.
(``Spherical'' means that it is an eigenfunction of $P_1$ with value $1$ at $e$ 
that depends only on $|x|$.) We see that also $P\Pfi(x,y) = \rho\, \Pfi(x,y)$.
The local limit theorem of \cite{Sa} says that
$$
p^{(n)}(x,y) \sim C\, \Pfi(x,y) \,\rho^n \, n^{-3/2}\,,\quad \text{as }\; n \to \infty\,.
$$
Thus, we get 
$$
H(x,y) = \frac{\Pfi(x,y)}{\Pfi(e,y)}.
$$
\begin{thm}\label{thm:radial} 
For isotropic $P$ as above, suppose that it has super-exponential moments: 
$\limsup_{d\to \infty} a_d^{1/d} = 0$.
Then the ratio limit compactification coincides with the $\rho\,$-Martin
compactification analytically. Geometrically, this is the end compactification, and
for $\xi \in \partial \T$,
$$
K(x,\xi|\rho) = H(x,\xi) = q^{-\hor(x,\xi)/2}.
$$
\end{thm}

\begin{proof}
It is known from \cite[Thm. 1.3]{Gou2} that under the super-exponential moment condition 
the $\rho\,$-Martin compactification coincides
geometrically with the end compactification.

From \eqref{eq:spherical}, we compute for any end $\xi \in \partial \T$
$$
\lim_{y \to \xi} H(x,y) = \lim_{y \to \xi} 
\frac{1+\frac{q-1}{q+1}d(x,y)}{1+\frac{q-1}{q+1}|y|}
q^{-d(x,y)/2+|y|/2} = q^{-\hor(x,\xi)/2}.
$$
This shows that the ratio limit compactification coincides geometrically with
the end compactification. Thus, what is left to show is that 
\begin{equation}\label{eq:K}
K(x,y|\rho) = q^{-\hor(x,\xi)/2}.
\end{equation}
For this purpose, we use Proposition \ref{pro:compare}. 
Since $P = \wh P(P_1)$, equation \ref{eq:rhoP} implies that 
$\HH^+\bigl(P_1,\rho(P_1)\bigr) \subset \HH^+(P,\rho)$. The Green function for $P_1$
is very well known, see e.g. \cite[Lemma 1.12]{Wbook}.
In particular,
$$
G_{P_1}\bigl(x,y|1/\rho(P_1)\bigr) = \frac{2q}{q-1} q^{-d(x,y)/2},
$$
so that \eqref{eq:K} holds for $P_1\,$, whose Martin compactification is of
course again the end compactification of $\T$.

The spherical function $\Pfi(x)= \Pfi(x,e)$ is in  
$\HH^+\bigl(P_1,\rho(P_1)\bigr)$. 
We now consider the Doob transforms
$Q_1 = P_1^{\Pfi}$ and $Q = P^{\Pfi}$, that is, their respective matrix elements are
$$
q_1(x,y) = \frac{p_1(x,y)\Pfi(y)}{\rho(P_1)\,\Pfi(x)} \AND 
q(x,y) = \frac{p(x,y)\Pfi(y)}{\rho\,\Pfi(x)}
$$
Then $f \in \HH(P,\rho)$ if and only if $f/\Pfi \in \HH(Q,1)$, and analogously for 
$P_1$ and $Q_1$. This implies that  $\HH\bigl(Q_1,1\bigr) \subset \HH(Q,1)$.
In view of Proposition \ref{pro:compare}, we need to verify that the 1-Martin 
boundary $\partial \T$ is Dirichlet regular for $Q_1$. By {\sc Cartwright et al.}  \cite{CSW}, for this it is necessary
and sufficient that the Green kernel of $Q_1$ vanishes at infinity, that
is,
$$
\lim_{|x| \to \infty} G_{Q_1}(x,e|1) = 0. 
$$
Now, $G_{Q_1}(x,y|1) =  G_{P_1}\bigl(x,y|1/\rho(P_1)\bigr)\Pfi(y)/\Pfi(x)$. Thus,
$$
G_{Q_1}(x,e|1) = \frac{2q}{q-1} \bigg/ \Bigl( 1 + \frac{q-1}{q+1}|x|\Bigr) \to 0\,,
\quad \text{as }\; |x| \to \infty\,. 
$$
This concludes the proof.
\end{proof}

\begin{rmk}\label{rem:iso} It seems likely that the super-exponential moment condition
may be relaxed here. {\sc Cartwright and Sawyer} \cite{CaSa} have 
shown that $\HH^+(P,1) = \HH^+(P_1\,,1)$  
for arbitrary isotropic $P$, as long as it is irreducible. 
However, the $1$-Martin compactification  is the end compactification (and thus coincides analytically with the one of $P_1$) only under additional hypotheses, such as
first moment, i.e., $\sum_d d\,a_d < \infty\,$. Otherwise, the $1$-Martin boundary can contain
further, non-minimal elements. \cite{CaSa} contains no analogous result at the critical value 
$\rho$, where also the Martin boundary might behave more ``critically''. On the other hand, 
isotropic random walks on $\T$ are a rather special case, where more general results may hold.
\end{rmk}

\medskip

\noindent
\textbf{B. Nearest neighbour random walk on free groups}
\\[3pt]
Let $\Ff = \Ff_s$ be the free group on $s$ free generators $a_1\,,\dots, a_s$, and write
$a_{-i} = a_i^{-1}$. Let $I= \{\pm 1, \dots, \pm s \}$ and $S = \{ a_i: i\in I\}$.
Recall that every element $x$ of $\Ff$ can be written as a reduced word over $S$,
\begin{equation}\label{eq:word}
x= a_{i_1} a_{i_2} \cdots a_{i_k}\,,\quad i_l \in I\,,\; i_l \ne -i_{l-1}\,.
\end{equation}
The length $|x|$ of $x$ is $k$, and when $k=0$, this is the empty word $e$, which is 
the group identity. The Cayley graph of $\Ff$ with respect to $S$ is the tree 
$\T = \T_{2s-1}\,$: its vertex set is $\Ff$, and $x,y \in \Ff$ are connected
by an edge whenever $x^{-1}y \in S$. Thus, the natural geometric compactification
$\wh F$ of $\Ff$ is the end compactification of $\T$, with boundary 
$\partial \Ff = \partial \T$.

We now let $\mu$ be a probability measure on $\Ff$ whose support is
$\supp(\mu) = \{ e \} \cup S$. This nearest neighbour case serves as a 
warm-up for the next sub-section; it might be omitted but may be
instructive. The local limit theorem and the involved
$\rho\,$-harmonic function have been studied in detail by {\sc Gerl and Woess}  \cite{GeWo}. One has
\eqref{eq:local} with $\alpha = 3/2$, and $\rho < 1$ by non-amenability of $\Ff$.
We subsume those facts which are needed here. (The notation is slightly modified.)
By group invariance, $G(x,y|z) = G(e,x^{-1}y|z)$.
The function $G(z)= G(x,x|z)$ is solution of an implicit equation, 
which leads to a formula for $\rho$. Set $F_i(z) = F(e,a_i|z)$ for $i \in I$.
If $x \in \Ff$ has the reduced representation \eqref{eq:word} then 
$$
F(e,x|z) = F_{i_1}(z) \cdots F_{i_k}(z).
$$
The ends of the tree $\T$ which is the Cayley graph of $\Ff$ can be written as
infinite words
\begin{equation}\label{eq:infword}
\xi = a_{j_1} a_{j_2} a_{j_3} \cdots\,, \quad j_l \in I\,,\; j_l \ne -j_{l-1}\,,
\end{equation}
and the $n^{\textrm{th}}$ vertex on the geodesic $\pi(e,\xi)$ is 
$x_n= a_{j_1} a_{j_2} \cdots a_{j_n}\,$. For any $t \ge \rho$, the $t$-Martin compactification
is $\partial \T$. This goes back to {\sc Dynkin and Malyutov} \cite{DyMa} 
and {\sc Cartier} \cite{Ca}.

With $\xi$ as above, if $x \in \Ff$ has reduced representation \eqref{eq:word},
then there is a maximal index $m = m(x,\xi) \le k$ such 
that $j_1 = i_1\,,\dots, j_m=i_m$. Then $x_m = x \wedge \xi$ in the above description of
confluents in the geometry of the tree. Then the Martin kernel at $\xi$ is
$$
K(x,\xi|t) = \frac{F_{-i_k}(1/t)F_{-i_{k-1}}(1/t) \cdots F_{-i_{m+1}}(1/t)} 
                  {F_{i_1}(1/t)F_{i_2}(1/t) \cdots F_{i_m}(1/t)}
$$
It is always minimal.
The analysis of \cite{GeWo} yields that in a neighbourhood of the principal singularity 
$r = 1/\rho$, for $z \in \C \setminus [r\,,\,\infty)$, one has Puiseux series expansions 
of the form
\begin{equation}\label{eq:expand}
G(z) = \alpha_0 - \beta_0 \sqrt{r-z} + \textit{h.o.t.} \AND
F_i(z) = \alpha_i - \beta_i \sqrt{r-z} + \textit{h.o.t.}\,,
\end{equation}
where $\alpha_i, \beta_i > 0$ for $i \in I \cup \{0\}$, and \textit{h.o.t.} stands for
series of ``higher order terms'' of the form $C \cdot (r-z)^q$, where $C$ is a constant and
$q$ is a rational number with $q > 1/2$, and the appearing exponents form a discrete
subset of $\Q$. As a matter of fact, in the present case, $q$ is always an integer
multiple of $1/2$.  From this, one gets with $x$ as in \eqref{eq:word}, $\xi$ as
in \eqref{eq:infword} and $m=m(x,\xi)$ as above that
\begin{equation}\label{eq:Knn}
 K(x,\xi|\rho) = \frac{\alpha_{-i_k}\alpha_{-i_{k-1}} \cdots \alpha_{-i_{m+1}}} 
                  {\alpha_{i_1}\alpha_{i_2} \cdots \alpha_{i_m}}.
\end{equation}
On the other hand, again for $x$ as in \eqref{eq:word}, we get for $z$ near $r$ as 
above the Puiseux series expansion
\begin{equation}\label{eq:alphabeta}
\begin{gathered}
G(e,x|z) = F_{i_1}(z)\cdots F_{i_k}(z)G(z)
= \alpha(x) - \beta(x) \sqrt{r-z} + \textit{h.o.t.}\,,\quad \text{where}\\
\alpha(x) = \alpha_0\, \alpha_{i_1}\alpha_{i_2} \cdots \alpha_{i_k} \AND 
\beta(x) = \alpha(x) \,\gamma(x)\,, \quad \text{with} \quad
\gamma(x) = \frac{\beta_0}{\alpha_0}+
\sum_{l=1}^k \frac{\beta_{i_l}}{\alpha_{i_l}}.
\end{gathered}
\end{equation}

\begin{cor}\label{cor:nn}
For an aperiodic nearest neighbour random walk on the free group $\Ff$,
the ratio limit kernel is $H(x,y) = \beta(x^{-1}y)/\beta(y)$. 
The ratio limit compactification coincides analytically with the
$\rho\,$-Martin compactification, which geometrically is the end compactification $\wh \Ff = \Cc_{\textrm{ends}}(\T)\,$.
\end{cor}

\begin{proof} The expansion \eqref{eq:alphabeta} yields the following local limit theorem,
see \cite[Thm. 2]{GeWo}:
$$
p^{(n)}(e,x) \sim \frac{1}{2\sqrt{\rho\,\pi}}\, \beta(x)\, \rho^n\, n^{-3/2}\,. 
$$
This yields the stated form of the ratio limit kernel.
We need to show that for every $x \in \Ff$  and every end $\xi$, 
$$
\lim_{y \to \xi} \frac{\beta(x^{-1}y)}{\beta(y)} = K(x,\xi|\rho).
$$
Again, suppose that $x$ is as in \eqref{eq:word}. Let $m = m(x,\xi) \in \{0, \dots, k\}$. 
We now write the reduced representation of $y$ as
$y = a_{j_1} a_{j_2} \cdots a_{j_n}\,$. In principle, the indices $j_1\,,\dots, j_n$
vary with $y$, but $y \to \xi$ means that $\bar m = m(y,\xi) \to \infty$, so that the initial
piece $a_{j_1} a_{j_2} \cdots a_{j_{\bar m}}$ coincides with the initial word of
$\xi$ of the same length. Furthermore, we will have $\bar m \ge m(x,\xi)$ when $y$
is close to $\xi$ in the geometric compactification $\Cc_{\textrm{ends}}(\T)$. 
For such $y$, recalling that $x_m$ is the $m^{\textrm{th}}$ element on $\pi(e,x)$,
$$
x^{-1}y = 
\underbrace{a_{-i_k}a_{-i_{k-1}}\cdots a_{-i_{m+1}}}_{\dps x^{-1}x_m} 
\underbrace{a_{j_{m+1}}a_{j_{m+2}} \cdots a_{j_n}}_{\dps x_m^{-1}y}\,.
$$
Note that by \eqref{eq:Knn} we then have $\alpha(x^{-1}y)/\alpha(y) = K(x,\xi|\rho)$. Therefore, 
when $\bar m \ge m$,
$$
H(x,y) = K(x,\xi|\rho) \, \frac{\gamma(x^{-1}y)}{\gamma(y)}.
$$
We then get $\gamma(x^{-1}y) - \gamma(y) = \gamma(x^{-1}x_m) - \gamma(x_m)\,$,
while $\gamma(y) \to \infty\,$. Thus, 
$$
\frac{\gamma(x^{-1}y)}{\gamma(y)} \to 1\,,\quad \text{as }\; y \to \xi.
$$
This concludes the proof.
\end{proof}

\medskip

\noindent
\textbf{C. Bounded range random walk on free groups}
\\[3pt]
The result of this sub-section generalises the previous one. 
\begin{thm}\label{thm:finrange} 
Suppose that the probability measure $\mu$ on $\Ff$ has finite support $S$ which generates
$\Ff$ as a semi-group and contains the group identity. 
Then the ratio limit compactification coincides with $\rho\,$-Martin
compactification analytically. Geometrically, it is the end compactification.
\end{thm}

This needs some preparation. 
From {\sc Lalley} \cite{La}, it is known that the random walk satisfies again a local limit theorem
\eqref{eq:local} with $\alpha = 3/2$, and $\rho < 1$ by non-amenability of $\Ff$.
We shall use a mix of the methods of \cite{De} and its extensions  
by {\sc Picardello and Woess} \cite{PiWo1}, and of \cite{La}, 
compare with \cite[\S 19.B and \S 26.A]{Wbook}.

For $n \in \N$, let let $B_n = \{ x : |x| \le n\}$ be the ball of radius
$n$ around the identity (root) with respect to the metric of the tree $\T$ 
which is the Cayley graph of the group.  
For any $y \in \Ff$, the set $yB_n$ is the ball
of radius $n$ centred at $y$.
Let $R = \max \{ |x| : x \in S \}$, and let $B = B_R\,$. For any set 
$A \subset \Ff$, we consider the stopping time plus associated 
probability generating function
$$
\st^{A} = \inf \{ n \ge 0 : X_n \in yA \} \AND 
F^{A}(u,v|z) = \sum_{n=0}^{\infty} \Prob[\st^{A} = n\,,\; X_n = v | X_0=u]\,z^n\,,
$$
where $u,v \in \Ff$. 
For the simple proof of the following, see the above references.

\begin{lem}\label{lem:intermed}
 If $x_0\,, x_1 \in \Ff$ are distinct and $y \in \pi(x_0\, ,x_1)$ then the random 
 walk starting at $x_0$ must pass through $yB$ in order to reach $x_1\,$. Thus,
 $$
 F(x_0\,,x_1| z)  = \sum_{v \in yB} F^{yB}(x_0\,,v|z)F(v,x_1|z).
 $$
\end{lem}

Next, for $A \subset \Ff$ as above, let 
$$
F_A(u,u'|z) = \sum_{n=0}^{\infty} \Prob[\st^{\Ff \setminus A} > n\,,\; X_n = u' | X_0=u]\,z^n\,,
$$
where $u,u' \in A$.  There is $N \ge R$ such that 
$$
F_{B_N}(u,u'|z) > 0 \quad \text{for all }\; u,u' \in B=B_R \;\text{ and all }\; 
z \in (0\,,\,r]\,,
$$
where, as above, $r =1/\rho$. (This is a simple observation: 
there must be a sequence $u=u_0\,,u_1\,\dots ,u_k = v$ such that $p(u_{i-1}\,,u_i) > 0$ for 
all $i$. We take $N$ large enough such that for any choice of $u, v \in B$ there is
such a sequence which is entirely contained in $B_N$.)

For  $x, y \in \Ff$,
we define the square matrix, resp.  (column) vectors
\begin{equation}\label{eq:Fb}
\begin{gathered}
\Fb(x,y|z) = \bigl( F^{yB}(xu,yv|z) \bigr)_{u, v \in B}\,,\\ 
\fb(x,y|z) = \bigl(F^{yB}(x,yu|z) \bigr)_{u \in B} \AND
\gb(x,y|z) = \bigl(G(xu,y|z)\bigr)_{v \in B}.
\end{gathered}
\end{equation}
We now let $D =  N+2R+1$ and consider the set $W_D= \{w \in \Ff : |w| = D\}$
of all elements (words) in $\Ff$ with length $D$. We observe that when
$d(x,y) = D$ then $w = x^{-1}y \in W_D$ and $\Fb(x,y|z) = \Fb(e,w|z)=:\Fb(w|z)$. 
Then the following is a consequence of Lemma \ref{lem:intermed}, see \cite{De}, \cite{PiWo1}.

\begin{lem}\label{lem:prod}
 Let $x, y$ in $\Ff$ and $u_0\,,u_1\,,\dots, u_n \in \pi(x,y)$
 such that $d(u_0\,,x) >R$, $d(u_n, y) > R$ and $d(u_k\,, x) = d(u_0\,,x) + k\,D\,$,
 so that $w_k = u_{k-1}^{-1}u_k \in W_D$ for $k=1, \dots, n$. Then for $z  \in (0\,,\,r]$ with
 $r=1/\rho$,
 $$
 G(x,y|z) = \bigl\langle \fb(x, u_0|z)\,,\, 
 \underbrace{\Fb(u_0\,,u_1|z) \cdots \Fb(u_{n-1}\,,u_n|z)}_{\dps 
 \Fb(w_1|z) \cdots \Fb(w_n|z)}\, \gb(u_n,y|z)\bigr\rangle.
$$
\end{lem}
Here, $\Fb(w_1|z) \cdots \Fb(w_n|z) \gb(u_n,y|z)$ is the product of $n$ square matrices
applied to the column  vector $\gb(u_n,y|z)$, and $\langle\cdot,\cdot\rangle$ is the ordinary inner product of column vectors indexed by $B$. 
Now 
set 
$$
\lambda_{z} = \min \{ F_{B_N}(u,u'|z) : u,u' \in B \} 
$$
Then 
$$
F^{wB}(u,wv|z) \ge F_{B_N}(u,u'|z)F^{wB}(u',wv|z) \ge \lambda_{z} \quad 
\text{for all }\; u, u', v \in B
\;\text{ and }\; z \in (0\,,\,r].
$$
In words, the first of the two inequalities comes from the fact that 
the random walk starting at $u \in B$ can reach $u' \in B$ with positive 
probability \emph{before} entering $wB$ at $wv$. In potential theoretic terms,
this can be interpreted as ``balayage'' or as a Harnack inequality.
What is important for us is that it tells us that all the matrices 
$\Fb(w|z)$, $w \in W_D$, $z \in (0\,,\,r]$, have their zeros disposed in columns,
and that in each non-zero column, the ratio of any two entries is bounded below 
by $\lambda_{z}$ when $z \in (0\,,\,r]$. 

For any vector $\vb \in (0\,,\,\infty)^B$, let 
$$
\Proj\,\vb = \frac{1}{\langle \vb, \uno\rangle} \,\vb
$$
be its projection onto the standard simplex over $B$ (all non-negative vectors whose coordinates sum up to $1$).
Then the above yields the following, see
\cite{De}, \cite{PiWo1} or \cite[\S 26.A]{Wbook}.

\begin{pro} \label{pro:convdir} Let $z \in (0\,,\,r]$ and $\underline{w}= (w_n)_{n \in \N}$
 be a sequence in $W_D\,$. Then there is a vector 
 $\wb_{\infty} = \wb_{\infty}(\underline{w},z) \in (0\,,\,\infty)^B$ with 
 $\langle \wb_{\infty}\,,\uno\rangle = 1$ such that uniformly for any 
 sequence of non-zero vectors  $\ab_n \in [0\,,\,\infty)^B$,
 $$
 \lim_{n \to \infty} \Proj\, \Fb(w_1|z) \cdots \Fb(w_n|z)\,\ab_n = \wb_{\infty}.
 $$
\end{pro}

The reason is that for each $w\in W_D\,$, the mapping
$\ab \mapsto \Proj\, \Fb(w|z)\,\ab$ is a contraction of the standard simplex
over $B$ with Lipschitz constant $\ell(\lambda_z) < 1$ that maps the simplex into
its interior.

It is known from the cited references that the $\lambda$-Martin compactification for $\lambda \ge \rho$ is always the end compactification of the tree, and that each end is a minimal
boundary element. For our purpose, we need the above material in
order to identify the Martin kernels for $\lambda = \rho$, i.e., $z=r$, as follows. Let $\xi \in \partial \T$,
and let $\pi(e,\xi) = [e=x_0\,,x_1\,, x_2\,,\dots]$. Let $u_n =  x_{nD}$ and, with the group operation of $\Ff$,  $w_n = u_{n-1}^{-1}u_{n}\,$. 
 We obtain a sequence 
$\underline{w} = \underline{w}(\xi)$ in $W_D$. Following Proposition \ref{pro:convdir},
we let for $k \in \N$
$$
\wb_{k,\infty} =  \wb_{k,\infty}(\xi) = \lim_{n \to \infty} 
\Proj\, \Fb(w_{k+1}|r) \cdots \Fb(w_n|r)\,\ab_n\,,
$$
which is independent of the specific choice of the positive vectors $\ab_n\,$.
Using Lemma \ref{lem:prod} and Proposition \ref{pro:convdir}, we now obtain
the following.

\begin{cor}\label{cor:Martin}
With the foregoing notation, in particular 
$\pi(e,\xi) = [e=x_0\,,x_1\,,\dots]$ for $\xi \in \partial \T$, 
let $x \in \Ff \equiv \T$, and let $k$ be 
such that $u_k=x_{kD} \in \pi(x \wedge \xi, \xi)$ and $d(u_{k}, x) > R$.
Then 
$$
\lim_{\T \ni y \to \xi} K(x,y|\rho) = 
\frac{\bigl\langle\fb(x, u_{k}|r), \wb_{k,\infty}\bigr\rangle}
     {\bigl\langle\fb(e, u_{k}|r), \wb_{k,\infty}\bigr\rangle} = K(x,\xi|\rho).
$$
\end{cor}

\begin{proof}[\rm \textbf{Proof of Thm. \ref{thm:finrange}}]
We need to describe the kernel $H(x,y)$ arising from the local limit
theorem of \cite{La} at least when $d(x,y)$ is large, and we need to
show that $H(x,y) \to K(x,\xi|\rho)$ when $y \to \xi \in \partial \T$.

The main point is that by \cite{La}, for $z$ near $r$, there is 
once more a Puiseux expansion 
$$
G(x,y|z) = G(x,y|r) - \beta(x,y)\, \sqrt{r-z} + h.o.t.,
$$ 
where $\beta(x,y)= \beta(e,x^{-1}y) > 0$. Then
$H(x,y) = \beta(x,y)/\beta(e,y)\,$.
It follows from \cite{La} (see also the exposition in \cite[\S 26.A]{Wbook})
that also the non-vanishing entries of all the matrices $\Fb(w|z)$, as well as 
the entries of the vectors $\fb(x,w|z)$, where $d(x,w) > R$, have 
Puiseux expansions of the same form. That is, in multidimensional notation we 
can expand
$$
\begin{aligned}
\Fb(w|z) &= \Fb(w|r) \;\, - \sqrt{r-z}\,\,\Bb(w) \;\,+ h.o.t., \quad w \in W_D\,,\\ 
\fb(x,y|z) &= \fb(x,y|r) - \sqrt{r-z}\,\bb(x,y) + h.o.t., 
\quad x,y \in \Ff\,,\; d(x,y) > R \,,\\
\gb(x,y|z) &= \gb(x,y|r) - \sqrt{r-z}\,\wt \bb(x,y) + h.o.t, \quad x,y \in \Ff\,.
\end{aligned}
$$
Here, $\Bb(w)$ is a non-negative matrix indexed by $B \times B$,
and $\Bb(w)$ is strictly positive in the same entries as $\Fb(w|r)\,$. 
Furthermore, the
non-negative $B$-indexed vectors $\fb(x,y|r)$ and $\bb(x,y)$ are strictly
positive in the same entries, while $\gb(x,y|r)$ and $\wt \bb(x,y)$
are positive in all entries.

\smallskip

We now use the same notation as in Corollary \ref{cor:Martin}. If $y \to \xi$
then $n = n_y \to \infty$, where $n_y$ is the largest integer such that
$u_n = x_{nD} \in \pi(e,y)$ and $d(u_{n}\,,y) > R$. In particular, we shall
have $n > k$. Using Lemma \ref{lem:prod}, we obtain
\begin{equation}\label{eq:beta}
\begin{aligned}
\beta(x,y) 
&= \bigl\langle \fb(x,u_k|r)\,,\,\Fb(w_{k+1}|r)\cdots\Fb(w_n|r)\,\wt \bb(u_n,y)\bigr \rangle \\
&\quad + \bigl\langle \bb(x,u_k)\,,\,\Fb(w_{k+1}|r)\cdots\Fb(w_n|r)\,\gb(u_n,y|r)\bigr \rangle \\
&\qquad + \sum_{i=k+1}^n \bigl\langle \fb(x,u_k|r)\,,\,
               \Fb(w_{k+1}|r)\cdots \Bb(w_i) \cdots \Fb(w_n|r)\, \gb(u_n,y|r)\bigr \rangle\,,
\end{aligned}
\end{equation}
where more precisely, $\Fb(w_{k+1}|r)\cdots \Bb(w_i) \cdots \Fb(w_n|r)$
is obtained from the matrix product $\Fb(w_{k+1}|r)\cdots\Fb(w_n|r)$ by replacing
the $i^{\,\textrm{th}}$ factor $\Fb(w_i|r)$ by $\Bb(w_i)$.
There is the analogous formula for $\beta(e,y)$, where one just needs to
replace $x$ by $e$. By a slight abuse of notation, we write
$\Proj \, \beta(x,y)$ for the expression where in all the inner products
of \eqref{eq:beta},
the vectors appearing in the second variable are replaced by their projection
onto the standard simplex. We choose $m(n) \ge k$ such that $m(n) \to \infty$ and
$m(n)/n \to 0$ (for example, $m(n) = \max \{ \lfloor \log n \rfloor, k \}$).
Then
$$
\Proj\, \Fb(w_{k+1}|r)\cdots \Fb(w_{m(n)}|r) \cdots \Bb(w_i) 
\cdots \Fb(w_n|r) \,\gb(u_n,y|r) \to \wb_{k+1,\infty} \quad 
\text{for all }\; i \!>\! m(n).
$$
Recall that this convergence is uniform in whatever non-negative vector 
appears on the right of $\Fb(w_{m(n)}|r)$.
Therefore, as $y \to \xi$, i.e., $n \to \infty$,
$$
\begin{aligned}
\frac{1}{n}  \sum_{i=m(n)+1}^{n} \bigl\langle \fb(x,u_k|r)\,,\,\Proj\,
    \Fb(w_{k+1}|r)\cdots \Bb(w_i) \cdots \Fb(w_n|r)\, \gb(u_n,y|r)\bigr \rangle&\\
\quad \sim   \frac{n-m(n)}{n} \bigl\langle \fb(x,u_k|r)\,,\,\wb_{k+1,\infty}\bigr \rangle\to  \bigl\langle \fb(x,u_k|r)\,,\,\wb_{k+1,\infty}\bigr \rangle&
\end{aligned}
$$
On the other hand, since $m(n)/n \to 0$,
$$
\frac{1}{n}\sum_{i=k+1}^{m(n)} \bigl\langle \fb(x,u_k|r)\,,\,\Proj\,
               \Fb(w_{k+1}|r)\cdots \Bb(w_i) \cdots \Fb(w_n|r)\, \gb(u_n,y|r)\bigr \rangle
\to 0.
$$
Also the extra first two terms of $\Proj\,\beta(x,y)$ divided by $n$ tend to $0$. All this is also valid for $e$ in the place of $x$. 
We see that
$$
\frac{1}{n} \, \Proj\, \beta(x,y) \to 
\bigl\langle \fb(x,u_k|r)\,,\,\wb_{k+1,\infty}\bigr \rangle \AND
\frac{1}{n} \, \Proj\, \beta(e,y) \to 
\bigl\langle \fb(e,u_k|r)\,,\,\wb_{k+1,\infty}\bigr \rangle\,.
$$
Taking quotients and comparing with Corollary \ref{cor:Martin}, we see
that $H(x,y) \to K(x,\xi|r)$ as $y \to \xi$.
\end{proof}

We remark here that the hypothesis that $\supp(\mu)$ contains the identity
can be replaced without substantial change by aperiodicity. Recall that this means that
$p^{(n)}(e,e) > 0$ for all but finitely many $n$, or in group theoretic terms,
that $\supp(\mu)$ is not contained in a coset of a proper normal subroup
of $\Gamma$ (in our case, $\Ff$). Also, Theorem \ref{thm:finrange} holds 
without substantial change of the proof for virtually free groups.

\section{Hyperbolic groups}\label{sec:hyperbolic}

We briefly recall the basic definition of hyperbolicity in the sense of
Gromov \cite{Gr}. Let $(\Xx,d)$ be a geodesic metric space, i.e., for 
any pair of points $x,y \in \Xx$, there is a (not necessarily unique) 
geodesic $\pi(x,y)$, that is, an isometric embedding 
$[0\,,\,d(x,y)] \hookrightarrow \Xx$ which maps $0$  to $x$ and $d(x,y)$ to $y$.
In our situation, $\Xx$ will carry the structure of a locally finite, connected 
graph and $d$ will be the graph metric. In this case, we replace the real interval
$[0\,,\,d(x,y)]$ by its integer counterpart $[0\,,\,d(x,y)]_{\Z} = \{ 0, 1, \dots, d(x,y)\}$.

The metric space is called \emph{hyperbolic,} if it is $\delta$-hyperbolic for some 
$\delta \ge 0$ (possibly large): if $a, b, c \in \Xx$ and 
$\pi(a,b)$, $\pi(b,c)$, $\pi(c,a)$ are geodesics between the respective points (the sides
of a triangle with vertices $a,b,c$) then for every $x \in \pi(a,b)$ there is 
$y \in \pi(b,c)\cup \pi(c,a)$ such that $d(x,y) \le \delta$. 
The most basic examples are provided by trees, where $\delta =0$. For all our purposes,
it will be convenient to assume without loss of generality that $\delta \in \N_0$ 
(non-negative integer).

We remark here that this implies, among many other facts, that any two geodesics 
connecting the same two points $x$ and $y$ are at Hausdorff distance at most $\delta$,
and we let $\Pi(x,y)$ denote the union of all those geodesics, a kind of ``slim sausage''.

A finitely generated group $\Gamma$ is called hyperbolic, if its Cayley graph
with respect to some finite, symmetric set of generators is hyperbolic.
This does not depend on the specific generating set, up to a change of
$\delta$. Basic examples are free groups and co-compact Fuchsian groups.
The entire theory will not be repeated here. For the present purpose, the
exposition in \cite[\S 22]{Wbook} plus the references given there will 
suffice. Besides very basic cases (virtually cyclic groups), all infinite 
hyperbolic groups are non-amenable.

A locally finite hyperbolic graph $\Xx$, resp. finitely generated hyperbolic group
$\Gamma$ has its hyperbolic compactification $\Cc_{\textrm{hyp}}(\Xx)$, resp. $\Cc_{\textrm{hyp}}(\Gamma)$.
It was shown by {\sc Ancona}~\cite{An} that under natural assumptions on $P$
on a hyperbolic graph $\Xx$ (bounded range, uniform irrecducibility; see \cite[\S 27]{Wbook}),
$\Cc_{\textrm{hyp}}(\Xx)$ is a geometric realisation of the $t$-Martin compactification for positive
$t > \rho$. In the group case this has been progressively strengthened in  papers by {\sc Gou\"ezel and Lalley}:

\begin{thm} [\cite{GL}, \cite{Gou1}.]\label{thm:gouezel}
Let $\Gamma$ be a non-amenable, finitely generated hyperbolic group, and $\mu$
a probability measure which induces an irreducible random walk. If $\mu$  is symmetric and is finitely supported,  
then the $t$-Martin compactification coincides with the 
$\Cc_{\textrm{hyp}}(\Gamma)$ for every $t \ge \rho$.

\smallskip

If in addition, $\mu$ is aperiodic, then the random walk satisfies 
a local limit theorem 
$$
p^{(n)}(x,y) \sim C \,\, \beta(x,y)\, \rho^n\, n^{-3/2} \quad \text{as }\;n \to \infty\,.
$$
\end{thm}

Regarding the Martin compactification, the noteworthy part is that it is also
valid at the critical value $t = \rho$. The crucial tool for this is the following,
whose part (a) was again first proved in \cite{An} for $z < r=1/\rho$ without 
requiring group-invariance, and then extended to $z=r$ in the group case in \cite{GL}, \cite{Gou1}, and finally \cite{Gou2}, including the strong inequality (b).

\begin{pro} [(Ancona inequalities).]\label{pro:anc}
Suppose that $\Gamma$  and $\mu$ are as in Theorem \ref{thm:gouezel}, and
consider a Cayley graph of $\,\Gamma$ with respect to a finite, symmetric
set of generators.
Then there are constants $C_{\!\text{\rm Anc}} \ge 1$ and $0 \le \alpha < 1$ such that 
the following holds for all $z \in [1\,,\,r]$.
\\[4pt]
\emph{(a)} For any geodesic
path $\pi(x,y)$ in the graph and any $w \in \pi(x,y)$, one has
$$
C_{\!\text{\rm Anc}}^{-1} \,G(x,w|z)G(w,y|z) \le
G(x,y|z) \le C_{\!\text{\rm Anc}}\,G(x,w|z)G(w,y|z)\,.
$$
\emph{(b)} For any quadruple of points $x, x', y, y'$ such that
$d\bigl(\Pi(x,x'), \Pi(y,y')\bigr) = n \ge 2\delta$, 
one has
$$
\left| \frac{G(x,y|z)G(x',y'|z)}{G(x,y'|z)G(x',y|z)} - 1 \right| 
\le C_{\!\text{\rm Anc}}\,G(x,w|z)G(w,y|z)\, \alpha^n.
$$
\end{pro}

In order to get a feeling for the last inequality, observe that for a nearest
neighbour random walk on a tree, one has $\alpha =0$. As a matter of fact, 
it is proved in \cite{Gou2} that Proposition \ref{pro:anc} and Theorem
\ref{thm:gouezel} are also valid when instead of finite support, one
assumes that $\mu$ has super-exponential moments, that is, 
$\sum_x a^{|x|}\,\mu(x) < \infty$ for all $a > 1$.

Irreducibilty plus group-invariance of the random walk yield a local
Harnack inequality: there is a constant $C_{\text{\rm Har}} > 1$ such that 
\begin{equation}\label{eq:har}
G(x',y|z) \le C_{\text{\rm Har}}^{d(x,x')} G(x,y|z) \quad \text{for all }\; x,x',y \in \Gamma
\;\text{ and }\; z \in [1\,,\,r]\,.
\end{equation}
Even without symmetry, the same also holds for $G(y,x|z)$ and $G(y,x'|z)$. 
This leads to the following generalisation of the first Ancona inequality.
For $x,y,w \in \Gamma$ with $\ell = d\bigl( w, \Pi(x,y) \bigr)$ and $z \in [1\,,\,r]$,
\begin{equation}\label{eq:genanc}
\begin{gathered}
C_{\ell}^{-1} \,G(x,w|z)G(w,y|z) \le
G(x,y|z) \le C_{\ell}\,G(x,w|z)G(w,y|z)\,, \\
\text{where}\quad
C_{\ell} = C_{\!\text{\rm Anc}}^{\,}\, C_{\!\text{\rm Har}}^{\,2\ell}\,.
\end{gathered}
\end{equation}

\begin{thm}\label{thm:hyp}
Let $\Gamma$ be a non-amenable, finitely generated hyperbolic group, and $\mu$
a finitely supported symmetric probability measure which induces an irreducible \& aperiodic random walk. 
Then the ratio limit compactification 
coincides with $\rho\,$-Martin compactification analytically.
\end{thm}

\begin{proof}
Let $G'(x,y|z)$ be the derivative
of the Green function with respect to $z$, where $|z| < r$. 
It is proved in \cite{GL} and \cite{Gou1} that for all $x, y \in \Gamma$,
there is $\beta(x,y) > 0$ such that
$$
G'(x,y|z) \sim \beta(x,y)\big/\sqrt{r-z} \quad \text{as }\; z \to r\,,\; 0< z < r\,.
$$
By working through the ``Tauberian'' last part of \cite{GL}, one learns that
symmetry and aperiodicity yield the asymptotics of Theorem \ref{thm:gouezel} above, 
with constant $C = \sqrt{r/\pi}$. Therefore the ratio limit kernel is the left-sided
limit
$$
H(x,y) = \lim_{z \to r-} \frac{G'(x,y|z)}{G'(e,y|z)}. 
$$
It is a well-known consequence of the \emph{resolvent equation} that
for $z \in (0\,,\,r)$ one has
$$
G'(x,y|z) = G^{(2)}(x,y|z)/z^2\,,\quad\text{where}\quad
G^{(2)}(x,y|z) = \sum_{v \in \Xx} G(x,v|z)G(v,y|z)\,.
$$
We now set
$$
\Phi(x,y|z) = \frac{G^{(2)}(x,y|z)}{G(x,y|z)}\,,\quad\text{so that}\quad
\frac{H(x,y)}{K(x,y|\rho)} = \lim_{z \to r-} \frac{\Phi(x,y|z)}{\Phi(e,y|z)}
=: \frac{\Phi(x,y|r-)}{\Phi(e,y|r-)}.
$$
Here, we always assume that $z$ is real, $z \in [1\,,\,r)$.
The proof will be complete once we have shown the following.
\\[6pt]
\underline{Claim.} \hspace*{5cm}
$\dps\lim_{|y| \to \infty} \frac{\Phi(x,y|r-)}{\Phi(e,y|r-)} = 1\,.$
\\[6pt]
\emph{Proof of the Claim\footnote{$\;$This was facilitated significantly by input from S\'ebastien Gou\"ezel.}.}
Let $\pi(e,y) =[e=y_0\,,y_1\,,\dots, y_n =y]$ be a geodesic from $e$ to $y$ in our 
Cayley graph, so that $n = |y|$. We set
$$
k(v) = k_y(v) = \max \bigl\{ k : d\bigl(y_k,\Pi(e,v)\bigr) \le \delta \bigr\}.
$$
Let $w \in \Pi(e,v)$ be such that $d(y_{k(v)},w) = d\bigl(y_{k(v)},\Pi(e,v)\bigr)$,
and let $\pi(e,v)$ be a geodesic that contains $w$. Using the geodesic triangle with
sides $\pi(e,w)$, $\pi(e,y_{k(v)})$ and a third side $\pi(w,y_{k(v)})$ (which has length $\le \delta$), one finds that $d\bigl(y_j,\pi(e,v)\bigr) \le 2\delta$ for 
$j \le k(v)$ and $\le \delta$ for all $j \le k(v)-\delta$. Also, $|w| \ge k(v) - \delta$.
Second, take a geodesic $\pi(y,v)$. If $k(v) < |y|$ then we must have 
$d\bigl(y_{k(v)+1}\,,\pi(v,y)\bigr) \le \delta$. Thus (even when $k(v)=|y|$) we have
\begin{equation}\label{eq:del}
d(y_{k(v)}\,,\bar v) \le \delta+1 \quad \text{for some }\; \bar v \in \pi(v,y). 
\end{equation}
which will be needed below. Consider the geodesic rectangle with sides
$\pi(y,v)$, $\pi(w,y_{k(v)})$ and $\pi\bigl(y_{k(v)}\,,y\bigr) \subset \pi(e,y)$ 
as well as $\pi(w,v) \subset \pi(e,v)$. The rectangle is $2\delta$-thin: 
every element on $\pi(y,v)$ is at distance at most $2\delta$ from one of the other three
sides, while any element on one of those three sides has distance at least $k(v)-\delta$
from the root (identity)~$e$. Hence,
$$
d\bigl(\Pi(y,v),e\bigr) \ge k(v) - 3\delta\,.
$$
Thus, when $k(v) \ge |x|+5\delta$, we have $d\bigl(\Pi(e,x), \Pi(y,v)\bigr) \ge 2\delta$,
and the strong Ancona inequality of Proposition \ref{pro:anc}(b) yields
\begin{equation}\label{eq:bound1}
\left| \frac{G(x,v|z)G(e,y|z)}{G(e,v|z)G(x,y|z)} - 1 \right| \le C_{\!\text{\rm Anc}}\, \alpha^{k(v)}.
\end{equation}
Combining \eqref{eq:bound1} with \eqref{eq:har}, we get for the given $y$ and 
geodesic $\pi(e,y)$ that
\begin{equation}\label{eq:bound}
\left| \frac{G(x,v|z)G(e,y|z)}{G(e,v|z)G(x,y|z)} - 1 \right| \le C(x)\, \alpha^{k(v)}
\quad \text{for all }\; x, v \in \Gamma
\;\text{ and }\; z \in [1\,,\,r]\,,
\end{equation}
where $C(x)$ depends only on $x$.

We now write for $z \in [1,r)$
$$
\Phi(x,y|z)-\Phi(e,y|z) = \sum_{v \in \Gamma} \frac{G(e,v|z)G(v,y|z)}{G(e,y|z)}
\left(\frac{G(x,v|z)G(e,y|z)}{G(e,v|z)G(x,y|z)} - 1 \right).
$$
Then \eqref{eq:bound} yields
$$
\bigl|\Phi(x,y|z)-\Phi(e,y|z)\bigr| \le  
C(x) \sum_{k=0}^{|y|} \alpha^k
\sum_{v: k(v)=k} \frac{G(e,v|z)G(v,y|z)}{G(e,y|z)}
$$
By Proposition \ref{pro:anc}(a), 
$$
G(e,y|z) \ge C_{\!\text{\rm Anc}}^{-1} G(e,y_{k(v)}|z) G(y_{k(v)}\,,y|z).
$$
By \eqref{eq:genanc} and, for the second inequality, \eqref{eq:del} 
$$
\begin{aligned}
G(e,v|z) &\le C_{\delta} \, G(e,y_{k(v)}|z) G(y_{k(v)},v|z) \AND\\
G(v,y|z) &\le C_{\delta+1} \, G(v,y_{k(v)}|z) G(y_{k(v)},y|z).
\end{aligned}
$$
We get for any $k \le |y|$ that with 
$\ol C = C_{\!\text{\rm Anc}}\,C_{\delta}\, C_{\delta+1}\,$,
$$
\begin{aligned}
\sum_{v: k(v)=k} \frac{G(e,v|z)G(v,y|z)}{G(e,y|z)} &\le 
\ol C \sum_{v: k(v)=k} G(y_k\,,v|z)G(v,y_k|z)\\ 
&\le \ol C\, G^{(2)}(y_k\,,y_k|z)
= \ol C\, G^{(2)}(e,e|z).
\end{aligned}
$$
\cite{Gou1} uses the notation $G^{(2)}(e,e|z) = \eta(r)$,
with $z \leftrightarrow r$, while our $r = 1/\rho$ is denoted $R$.
With $\ol C(x) = C(x) \,\ol C/(1-\alpha)$, we get 
$$
\left|\frac{\Phi(x,y|z)}{\Phi(e,y|z)} - 1\right| \le 
\ol C(x) \frac{G^{(2)}(e,e|z)}{\Phi(e,y|z)}.
$$
Now the crucial point is that following \cite[Lemma 3.20 and equation (3.15)]{Gou1},
$$
\Phi(e,y|z) \asymp |y|\,G^{(2)}(e,e|z) \quad \text{uniformly for }\; 
z \in [1\,,\,r)\;\text{ as }\; |y| \to \infty\,,
$$
i.e., the ratio is bounded above and below by uniform positive constants 
$A$ and $1/A$, respectively, when $|y|$ is large enough. Thus, for large $|y|$, 
$$
\left|\frac{\Phi(x,y|r-)}{\Phi(e,y|r-)} - 1\right| \le 
\ol C(x)\,\frac{A}{|y|}, 
$$
which tends to $0$ as $|y| \to \infty\,$. This concludes the proof of the
Claim and the Theorem.
\end{proof}

We note here that the last theorem does not fully cover the results on
free groups and trees of \S \ref{sec:trees}. Theorem \ref{thm:radial} does
not need finite range, and Theorem \ref{thm:finrange} does not need symmetry,
while symmetry of $\mu$ is a crucial tool for the local limit theorem on
hyperbolic groups stated in Theorem \ref{thm:gouezel}. In any case, it may be interesting 
to watch out how certain features of the respective proofs show up in different 
``disguise'' in each of them.

\section{Direct and Cartesian products}\label{sec:prod}

The ratio limit compactification adapts quite well to direct products.
In general, let $\Cc(\Xx_1)$ and $\Cc(\Xx_2)$ be compactifications
of the two discrete state spaces $\Xx_1$ and $\Xx_2\,$, with respective boundaries
$\partial \Xx_1$ and $\partial \Xx_2\,$.
Then $\Cc(\Xx_1) \times \Cc(\Xx_2)$ is the natural associated compactification of
$\Xx = \Xx_1 \times \Xx_2\,$. 
In this case, the boundary is
\begin{equation}\label{eq:prodbd}
\partial \Xx = \bigl(\partial \Xx_1 \times \partial \Xx_2 \bigr)
\cup \bigl( \Xx_1\times \partial \Xx_2) \cup \bigl( \partial \Xx_1\times \Xx_2).
\end{equation}
We write elements of the product space as $w_1w_2\,$, where $w_i \in \Cc(\Xx_i)$.
In the resulting topology, let $\bigl(y_1(n)y_2(n)\bigr)_{n \in \N}$ be 
a sequence in $\Xx$. Then, as $n \to \infty\,$, 
$$
\begin{aligned}
 y_1(n)y_2(n) \to \xi_1\xi_2 \in \partial \Xx_1 \times \partial \Xx_2
 \iff &y_i(n) \to \xi_i \;\text{ in }\; \Cc(\Xx_i) \;\text{ for }\; i=1,2\,,\\
 y_1(n)y_2(n) \to w_1\xi_2 \in \Xx_1 \times \partial \Xx_2
 \iff &y_1(n) = w_1 \;\text{for all but finitely many $n$, and}\\ 
      &y_2(n) \to \xi_2 \;\text{ in }\; \Cc(\Xx_2)\,,\\
  y_1(n)y_2(n) \to \xi_1 w_2 \in \partial \Xx_1 \times \Xx_2
 \iff &y_1(n) \to \xi_1 \;\text{ in }\; \Cc(\Xx_1), \; \text{ and}\\ 
      &y_2(n) = w_2 \;\text{for all but finitely many $n$.}
\end{aligned}
$$
Let us call this the \emph{product compactification} of the
given compactifications of $\Xx_1$ and $\Xx_2\,$, with the 
\emph{product boundary} \eqref{eq:prodbd}.
Now let $(\Xx_1\,, P_1)$ and $(\Xx_2\,,P_2)$ the state spaces plus 
irreducible and aperiodic transition matrices of two respective
Markov chains. Suppose that for $i =1,2$
$$
\lim_{n \to \infty} \frac{p_i^{(n)}(x_i\,,y_i)}{p_i^{(n)}(e_i\,,e_i)} 
= h_i(x_i\,,y_i)  \quad \text{for all}\; x_i,y_i \in \Xx\,.
$$
The direct product is the Markov chain on $\Xx_1 \times \Xx_2$ with
transition matrix $P=P_1 \otimes P_2$, where
$$
p(x_1x_2\,,y_1y_2) = p_1(x_1\,,y_1)p_2(x_2\,,y_2)\,.
$$
It is clear that it has a ratio limit 
\eqref{eq:rl} with 
$$
h(x_1x_2\,,y_1y_2) = h_1(x_1\,,y_1)h_2(x_2\,,y_2), 
$$
so that also the kernel $H(x_1x_2\,,y_1y_2)$ splits in the same
way. The associated ratio limit compactification of $\Xx_1 \times \Xx_2$
is not always the product compactification of the two ratio limit compactifications,
but it is a factor thereof. The following is quite obvious; see e.g. the 
``preamble'' on compactifications in \cite[\S 7.B]{WMarkov}, and recall
 that the ratio limit compactification of $\Xx_1 \times \Xx_2$ is the 
 minimal one which provides continuous extensions of the functions $H(x_1x_2\,,\cdot)$,
 where $x_1x_2 \in \Xx_1 \times \Xx_2\,$.

 \begin{lem}\label{lem:dirprod}
 Consider the extensions of $H_i(x_i\,,\cdot)$ to the ratio limit boundary $\Rbd_i$ of $\Xx_i\,$,
 $i=1,2$. For $\eta = \eta_1\eta_2$ and $\zeta = \zeta_1\zeta_2 \in 
 \bigl(\Rbd_1 \times \Rbd_2 \bigr) \cup \bigl( \Xx_1\times \Rbd_2) \cup \bigl( \Rbd_1\times \Xx_2)$,
 let 
 $$
 \eta \approx \zeta \iff 
 H_1(x_1\,,\eta_1)H_2(x_2,\eta_2) =  H_1(x_1\,,\zeta_1)H_2(x_2,\zeta_2)
 \quad \text{for all }\; x_1x_2 \in \Xx_1 \times \Xx_2\,.
 $$
 Then the ratio limit boundary $\Rbd$ of $P= P_1 \otimes P_2$ is the image of the product
 boundary of the two ratio limit compactifications with respect to the factor map
 of the equivalence relation $\approx$. In particular, the extension of
 $H(x_1x_2\,,\cdot)$ to $\Rbd$ is given by $H(x_1x_2,\xi) =  H_1(x_1\,,\eta_1)H_2(x_2,\eta_2)$,
 where $\eta_1\eta_2$ is a representative of the $\approx$-equivalence class $\xi$. 
 \end{lem}

In a certain sense more natural than direct products are \emph{Cartesian products.}
Given $P_i$ in $\Xx_i$ for $i=1,2$, write $I_i$ for the identity operator (or matrix)
over $\Xx_i$ and choose a parameter $s \in (0\,,\,1)$. On $\Xx = \Xx_1 \times \Xx_2\,$,
let 
$$
P = P_{s} = s \cdot P_1\otimes I_2 + (1- s)\cdot I_2 \otimes P_2\,.
$$
That is, the new Markov chain is such that first a coin is tossed where ``heads''
comes up with probability $s$, and in that case, a step is performed according to
$P_1$ in tb first coordinate, while the second coordinate remains unchanged. And
with probability $1-s$, a step is performed according to $P_2$ in the second 
coordinate, while the first one remains unchanged. 

For example, simple random walk on $\Z^{d_1+d_2}$ arises as a Cartesian (and not direct)
product of the simple random walks on $\Z^{d_1}$ and $\Z^{d_2}$ with $s = d_1/(d_1+d_2)$.

In general, it does not appear to be completely straightforward that a ratio limit for 
each $P_i$ also implies one for the Cartesian products $P_s\,$. However, the 
following was proved by {\sc  Cartwright and Soardi}~\cite{CaSo}.

\begin{pro}\label{pro:car} Suppose that each $\Xx_i$ satisfies a local limit
theorem of the form
$$ 
p_i^{(n)}(x_i,y_i) \sim \beta_i(x_i,y_i)\, \rho_i^n\, n^{-\alpha_i} 
\quad \text{as }\;n \to \infty\,,
$$
where $\rho_i = \rho(P_i)$ for $i = 1,2$.
Then one has for $P = P_s$ that $\rho = \rho(P) = s\,\rho_1 + (1-s)\,\rho_2\,$, 
$$
p^{(n)}(x_1x_2\,,y_1y_2) \sim 
C\,\beta_1(x_1,y_1)\beta_2(x_2,y_2)\, \rho^n\, n^{-\alpha_1-\alpha_2}\,,
$$
where $C = \theta^{\alpha_1}(1-\theta)^{\alpha_2}\,$, with $\theta = s\rho_1/\rho$.   

In particular, in this case the ratio limit kernel for the Cartesian product
is also $H(x_1x_2\,,y_1y_2) = H_1(x_1\,,y_1)H_2(x_2\,,y_2)$, and the ratio limit
boundary is the same as for the direct product. 
\end{pro}
 
At this point, there is a natural question. Suppose that the $\rho_i\,$-Martin boundary
of $P_i$ coincides with the respective ratio limit boundary for $i=1,2$. Is this then 
also true for the $\rho$-Martin boundary for the direct, resp. Cartesian products?
So far, there is no general answer; the general problem lies in  $\rho$-Martin kernels
of the product which are not minimal $\rho$-harmonic functions. Some examples are known.
We present three of them, which illustrate different situations regarding the 
equivalence relation $\approx$ that appears in Lemma \ref{lem:dirprod}. All three 
examples are valid for direct as well as Cartesian products, and the answer to
the above question is ``yes''.

\begin{exa}\label{exa:1}
 Suppose that $\Xx_i = \Z^{d_i}$ and that the respective irreducible, aperiodic
 random walk is induced by a finitely supported probability measure $\mu_i\,$.
 Let $h_i(x) = \exp \langle c_i\,,x \rangle$ be the unique $\rho_i\,$-harmonic
 exponential on $\Z^{d_i}$. That is, the vector $c_i$ minimises 
 $c \mapsto \sum_{x} \mu_i(x) \exp \langle c\,,x \rangle$, where 
 $c \in \R^{d_i}$, and the value of that minimum is the corresponding spectral 
 radius $\rho_i\,$. Every positive $\rho_i\,$-harmonic function on $\Z^{d_i}$
 is a constant multiple of $h_i$, and in the ratio limit theorem, one
 has by a slight abuse of notation $h_i(x, y) = h_i(x-y)$. Thus,
 $H_i(x,y) = h_i(x)$ for $x \in \Z^{d_i}$, and 
 the ratio limit kernel $H(x_1x_2\,,y_1y_2) = h_1(x_1)h_2(x_2)$ of the direct 
 product also does not depend on $y_1y_2\,.$ That is, the equivalence relation
 $\approx$ of Lemma \ref{lem:dirprod} has a single equivalence class.
 
 In this case, one also knows that all the involved compactifications coincide
 analytically with the Martin compactifications at the respective spectral radii.
 \qed
 \end{exa}

\begin{exa}\label{exa:2}
Let $\Xx_1 = \T$, the regular tree with degree $q+1 \ge 3$, 
and $\Xx_2 = \Z$. On each of the two, we consider ``lazy'' simple random walk,
that is,
$$
\begin{aligned}
p_1(x_1\,,x_1) &= \tfrac12 \;\text{ and }\; 
p_1(x_1\,,y_1) = \tfrac{1}{2q+2} \;\text{ when }\; x_1 \sim y_1\,,\\
p_2(x_2\,,x_2) &= \tfrac12 \;\text{ and }\; 
p_2(x_2\,,y_2) = \tfrac{1}{4} \;\text{ when }\; x_2 \sim y_2\,,
\end{aligned}
$$
while all other transition probabilities are $0$.
Then $H_1(x_1\,,y_1) = \Phi(x_1\,,y_1)$, see \S \ref{sec:trees}.A, 
and $H_2(x_2\,,y_2) = 1$. The ratio limit compactification of $\Z$ is the one-point
compactification $\Z \cup\{\infty\}$,  and the ratio limit boundary of the product space
$\T \times \Z$ is $\Cc_{\textrm{ends}}{\T}\,$.  Convergence to the boundary of a sequence 
$\bigl( y_1(n)y_2(n) \bigr)_{n \in \N}$
is as follows, where $K_1(\cdot,\cdot|\rho_1)$ is the Martin kernel on $\T$ given by
\eqref{eq:K}
$$
\begin{aligned}
 y_1(n) \to \xi_1 \in \partial \T\,, y_2(n) \;\text{arbitrary} &\Rightarrow 
 H\bigl(x_1x_2\,, y_1(n)y_2(n) \bigr) \to K_1(x_1\,,\xi_1 | \rho_1)\,,\\
 y_1(n) = y_1 \in \T \;\text{for all }\; n \ge n_0\,,\; |y_2(n)| \to \infty &\Rightarrow 
 H\bigl(x_1x_2\,, y_1(n)y_2(n) \bigr) \to \frac{\Phi(x_1\,,y_1)}{\Phi(o_1\,,y_1)}.
 \end{aligned}
$$
In particular, by {\sc Crotti}~\cite{Cr}, the ratio limit compactification coincides
analytically with the $\rho$-Martin compactification; see \cite[Thm. 28.8]{Wbook}.
In that reference, the result is stated for Cartesian products; the proof 
carries over to direct products with some obvious modifications. \qed
\end{exa} 

\begin{exa}\label{exa:3}
Let $\Xx_i = \T_i$  be two regular trees with respective degrees $q_i+1 \ge 3$. 
On each of the two, we consider ``lazy'' simple random walk as above,
that is,
$$
p_i(x_i\,,x_i) = \tfrac12 \;\text{ and }\; 
p_i(x_i\,,y_i) = \tfrac{1}{2q_i+2} \;\text{ when }\; x_i \sim y_i\,,
$$
while all other transition probabilities are $0$.
Then $H_i(x_i\,,y_i) = \Phi_i(x_i\,y_i)$, the spherical function on the respective tree
as in \S \ref{sec:trees}.A. The ratio limit compactification of the direct or any
Cartesian product of the two random walks is the product compactification of the 
ratio limit compactifications of the two trees.
Convergence to the boundary of a sequence 
$\bigl( y_1(n)y_2(n) \bigr)_{n \in \N}$
is as follows, where $K_i(\cdot,\cdot|\rho_i)$ is the Martin kernel on $\T_i$ given by
\eqref{eq:K}
$$
\begin{aligned}
 &y_1(n) \to \xi_1 \in \partial \T_1\,,\; y_2(n) \to \xi_2 \in \partial \T_2 \\
 &\hspace*{2cm}\Rightarrow 
 H\bigl(x_1x_2\,, y_1(n)y_2(n) \bigr) \to K_1(x_1\,,\xi_1 | \rho_1)K_2(x_2\,,\xi_2 | \rho_2)\,,
 \\[3pt]
 &y_1(n) = y_1 \in \T_1 \;\text{for all }\; n \ge n_0\,,\; y_2(n) \to \xi_2 \in \partial \T_2\\ 
 &\hspace*{2cm}\Rightarrow  
 H\bigl(x_1x_2\,, y_1(n)y_2(n) \bigr) \to 
 \frac{\Phi_1(x_1\,,y_1)}{\Phi_1(o_1\,,y_1)}\, K_2(x_2,\xi_2|\rho_2)\,,\\[3pt]
 &y_1(n) \to \xi_1 \in \partial \T_1\,,\;  y_2(n) = y_2 \in \T_2 \;\text{for all }\; n \ge n_0\\
 &\hspace*{2cm}\Rightarrow  
 H\bigl(x_1x_2\,, y_1(n)y_2(n) \bigr) \to 
 K_1(x_1,\xi_1|\rho_1)\, \frac{\Phi_2(x_2\,,y_2)}{\Phi_2(o_2\,,y_2)}\,.
 \end{aligned}
$$
By \cite{PiWo3}, the ratio limit compactification coincides once more
analytically with the $\rho$-Martin compactification. Again, this holds
for direct products in the same way as for Cartesian products. \qed
\end{exa}

\section{Reduced ratio limit compactification}\label{sec:reduced}

The companion paper \cite{Dor} makes crucial use of the following variant of the 
ratio limit compactification. Let $\sim$ be the equivalence relation on
$\Xx$ such that 
\begin{equation}\label{eq:equiv}
y \sim y' \iff H(x,y) = H(x,y') \quad \text{for all }\; x \in \Xx\,.
\end{equation}
We denote by $\wt \Xx$ the set of equivalence classes, and 
by $\tilde y$ the equivalence class of $y \in \Xx$. Then the ratio limit kernel
decends to a kernel on $\Xx \times \wt \Xx$ by
$$
H_{\textrm{red}}(x,\tilde y) = H(x,y).
$$
\begin{dfn}\label{def:red-rlc}
The \emph{reduced ratio limit compactification\footnote{$\;$\cite{Dor} calls this the (ordinary) ratio limit compactification.}} $\Rcp(\wt \Xx)$ associated
with $\Xx$ and $P$ satisfying \eqref{eq:rl}
is the (up to homomorphism) unique compact Hausdorff space
which contains $\wt \Xx$ as a discrete, dense subset  and has the following
properties: 
\begin{itemize}
 \item for each $x \in \Xx$, the function $H_{\textrm{red}}(x, \cdot)$ extends 
 continuously to $\Rcp(\wt \Xx)\,$, and denoting the extended kernel also by $H_{\textrm{red}}$,  
 \item if $\xi, \eta \in \Rbd(\wt \Xx) = \Rcp(\wt \Xx) \setminus \wt \Xx$ are distinct,
 then there is $x \in \Xx$ such that $H_{\textrm{red}}(x,\xi) \ne H_{\textrm{red}}(x,\eta)$.
\end{itemize}
\end{dfn}

If $\wt \Xx$ is finite then it is already compact, and there is no ratio limit boundary
added to that space.  The proofs of the following facts are easy exercises.

\begin{lem}\label{lem:facts} \emph{(i)} If the equivalence relation \eqref{eq:equiv} is extended 
to all of $\Rcp(\Xx)$ via the extended ratio limit kernel on $\Xx \times \Rcp(\Xx)$, then 
the resulting factor space is $\Rcp(\wt \Xx)$.
\\[5pt]
\emph{(ii)} The factor map $\Xx \to \wt \Xx$ extends to a continuous surjection
$\Rcp(\Xx) \to \Rcp(\wt \Xx)$ which is one-to-one from $\Rbd(\Xx)$ into the reduced
ration limit compactification.
\\[5pt]
\emph{(iii)} If an equivalence class $\tilde y$ is infinite, then it has a unique
accumulation point $\xi \in \Rbd(\Xx)$, and $H(x,y) = H(x,\xi)$ for all $y \in \tilde y$.
\end{lem}

If $\Xx = \Gamma$ is a countable group and $P$ is a random walk induced by a probability
measure $\mu$, then 
\begin{equation}\label{eq:radical}
R_{\mu} = \{ y \in \Gamma : H(x,y) = H(x,e) \;\text{ for every }\; x \in \Gamma \}
\end{equation}
is a subgroup of $\Gamma$, see \cite{Dor}. Since the ratio limit  kernel $h(x,y)$ 
in \eqref{eq:rl} satisfies
$h(x,y)= f(x^{-1}y)$, where $f(x)= h(x,e)$, one gets that 
$$
\wt \Xx \;\bigl(=  \Gamma_{\textrm{red}}\bigr) = \Gamma/R_{\mu}\,.
$$
{\sc Elder and Rogers}~\cite{ER} have extended Avez' Theorem \ref{thm:avez}:  they show 
that for a symmetric, aperiodic random walk on an arbitrary finitely generated group,
the set $A_{\mu}$ of all $y \in \Gamma$ for which $p^{(n)}(e,y)/p^{(n)}(e,e) \to 1$
is an amenable subgroup of $\Gamma$. This implies at least in the symmetric case that 
$R_{\mu} \subset A_{\mu}$ is amenable.

\begin{pro}\label{pro:fix} Consider a probability measure $\mu$ on the finitely generated
group $\Gamma$ which induces an irreducible \& aperiodic random walk satisfying  
\eqref{eq:rl} and \eqref{eq:rhoH}. Suppose that $R_{\mu}$ is infinite, 
and that the associated element $\xi \in  \Rbd(\Xx)$ according to Lemma \ref{lem:facts}(iii)
is such that $x \mapsto H(x,\xi)$ is a \underline{minimal} $\rho$-harmonic function. 

Then $\Gamma$ fixes the boundary point $\xi$.
\end{pro}

\begin{proof}
Like on the $t$-Martin compactification, the group acts continuously on
the ratio limit compactifictation, and
the extended ratio limit kernel satisfies the cocycle identity
$$
H(gx,\xi) = H(x,g^{-1}\xi) H(g,\xi) \quad \text{for all }\, x, g \in \Gamma
$$
(and of course for every boundary element, not only the $\xi$ of the statement).

By our assumptions, $H(x,y) = H(x,e) = h(x,e) = f(x)$ for all $y \in R_{\mu}$, and 
as $|y| \to \infty$, we have $y \to \xi$ in the topology of the ratio limit
compactification. Thus, $f(x) = H(x,\xi)$ for all $x \in \Gamma$.
By \eqref{eq:rhoH}, the function 
$$
\check f(x) = f(x^{-1}) = h(e,x)
$$
satisfies $fP = \rho\cdot f$, where $P$ is the transition matrix of the random walk.
Using the cocycle identity, this can be rewritten as
$$
H(x,\xi) = \check f(x) = \frac{1}{\rho} \sum_{g \in \Gamma} \mu(g) H(gx,\xi) 
= \sum_{g \in \Gamma}\underbrace{\frac{\mu(g)H(g,\xi)}{\rho}}_{\displaystyle =: c_g} H(x,g^{-1}\xi)
$$
for every $x \in \Gamma$. Since $\sum_g c_g = 1$, the minimality assumption on $H(\cdot,\xi)$ yields 
that $H(x,g^{-1}\xi) = H(x,\xi)$ for every $x \in \Gamma$ and every $g$ in the support of $\mu$. Therefore 
$g\xi = \xi$ for every $g \in \supp(\mu)$, and consequently for every  $g \in \Gamma$.
\end{proof}

\begin{cor}\label{cor:appl} 
\emph{(a)} For random walks on trees and free groups as considered in theorems \ref{thm:radial}
and \ref{thm:finrange},  the  subgroup $R_{\mu}$ is trivial.
\\[5pt]
\emph{(b)} For random walks on non-amenable hyperbolic groups as considered in Theorem \ref{thm:hyp}, the 
subgroup $R_{\mu}$ is finite. It is trivial when the group is torsion-free.
\end{cor}

Statement (a)  can of course also be seen from directly from the respective form of the extended 
ratio limit kernel,  in particlar for the isotropic case. However, in the non-isotropic case, the present 
method is more convenient.
Note that in all cases of the corollary, the $\rho$-Martin boundary coincides with the minimal one.
Furthermore, in all those cases, the entire group cannot fix a single boundary element: this has several
different proofs, among which the author is best acquainted with \cite[Prop. 4]{Wfix}.

In examples \ref{exa:2} and \ref{exa:3}, think $\T$ (even degree) as the free group.
The repective random walks are of course induced be probabiliy measure on the
respective product groups. In Example \ref{exa:2}, we have $R_{\mu} = \{e_1\} \times \Z$,
and the reduced ratio limit compactification is $\Cc_{\textrm{ends}}(\T)$. In 
Example \ref{exa:3}, there is no reduction; $R_{\mu}$ is trivial.

\section{Final remarks}\label{sec:final}

The above material should be seen as a collection of first answers to
the question stated in the Introduction. The next step in 
the same direction would concern relatively hyperbolic groups,
see the local limit theorem of \cite{Dus}. Under the assumptions of
that work, there is again a local limit theorem of the same form as for
hyperbolic groups (see Theorem \ref{thm:gouezel} above). 
It is not so easy to lay hands on the ratio limit kernel $H(x,y)$ in those cases,
but one would expect that one also has that the ratio limit compactification
coincides analytically with the $\rho$-Martin compactification.

\begin{que} (a) Under which general conditions is it true that 
the ratio limit kernel is  the left-sided limit 
$H(x,y) =  G^{(2)}(x,y|r-)/G^{(2)}(e,y|r-)$~? 
\\[5pt]
(b) Under which general conditions does the corresponding
compactification coincide analytically (or only geometrically) 
with the $\rho$-Martin compactification~?
\end{que}

For relatively hyperbolic groups, in particular for free products of groups,
there also are local limit theorems of the form \eqref{eq:local}, but with 
$\alpha > 2$, see {\sc Cartwright}~\cite{Car} and 
{\sc Candellero and Gilch}~\cite{CaGi}. This may be more challenging.

\medskip

In the last decades there has not been much work on 
ratio limit  theorems for random walks on groups; \cite{ER} is one of the interesting
exceptions. On non-amenable groups,
local limit theorems have prevailed. In a certain sense, the cases considered
here might also be referred to as the ``local limit boundary''. Indeed, apart from very few
exceptions (e.g. radial random walks on trees), the author does not know of 
methods which provide a ratio limit theorem in the non-amenable environment without 
first proving a local limit theorem. This may indicate possibilities for future
research.

\medskip

More generally, let $t \ge \rho = \rho(P)$, and suppose that we have 
a positive $t$-harmonic kernel $h(x,y)$, that is, $Ph(\cdot,y) = t\cdot h(\cdot,y)$.
In our situation, harmonicity is two-sided, i.e., we also have 
$h(x,\cdot)P = t \cdot h(x,\cdot)$. For example, if $P$ is the transition matrix
of a random walk on a group $\Gamma$ induced by the probability measure $\mu$,
then we may look for a solution of the convolution equation $\mu *\sigma = t\cdot \sigma$
(or the two-sided version) and set $h(x,y) = \sigma(x^{-1}y)$. 
Then we can normalise by setting $H(x,y) = h(x,y)/h(e,y)$ and try to understand
the corresponding compactification. If $K(\cdot,\cdot|t)$ is the $t$-Martin kernel of
$P$, then for each $y \in \Xx$ there is a probability measure $\nu^y$ on 
$\Mbd_t(\Xx)$ such that 
$$
H(x,y) = \int K(x,\xi|t)\, d\nu^y(\xi)\,.  
$$
In the case when $\Mbd_t(\Xx)$ has only minimal boundary elements, the compactification induced
by $H$ will coincide analytically with the $t$-Martin compactification when
for every $\xi \in \Mbd_t(\Xx)\,$, one has that $\nu^y \to \delta_{\xi}$ weakly as $y \to \xi$
in $\Mcpt(\Xx)$.

\end{document}